\documentclass[11pt]{amsart}
\usepackage{latexsym}
\usepackage{amsfonts}
\usepackage{xcolor}
\usepackage{graphicx}

\newcommand{\field}[1]{\mathbb{#1}}
\newcommand{\C}{\field{C}}

\newcommand{\R}{\field{R}}

\newtheorem{defi}{Definition}[section]
\newtheorem{lem}[defi]{Lemma}

\newtheorem{theo}[defi]{Theorem}
\newtheorem{co}[defi]{Corollary}
\newtheorem{pr}[defi]{Proposition}
\newtheorem{re}[defi]{Remark}
\newtheorem{ex}[defi]{Example}

\font\tenmsy=msbm10

\def\Bbb#1{\hbox{\tenmsy#1}} 
\title[On the extension of bi-Lipschitz mappings ]{On the extension of bi-Lipschitz mappings  }
\makeatletter

\@addtoreset{equation}{section}
\makeatother
\author{Lev Birbrair}
\address[Lev Birbrair]{ Departamento de Matem\'atica\\ Universidade Federal do Cear\'a, Av.
Humberto Monte, s/n Campus do Pici  60440-900\\ Brasil.}
\email{lev.birbrair@gmail.com}

\author{Alexandre Fernandes}
\address[Alexandre Fernandes]{Departamento de Matem\'atica\\ Universidade Federal do Cear\'a, Av.
Humberto Monte, s/n Campus do Pici  60440-900\\ Brasil.}
\email{alex@mat.ufc.br}

\author{Zbigniew Jelonek}
\address[Zbigniew  Jelonek]{Instytut Matematyczny\\
Polska Akademia Nauk\\
\'Sniadeckich 8, 00-656 Warszawa\\
Poland }
\email{najelone@cyf-kr.edu.pl}


\subjclass{14 R 10, 14 P 15, 14 P 10}
\thanks{Alexander Fernandes and Zbigniew Jelonek are partially supported by the grant of Narodowe Centrum Nauki number 2019/33/B/ST1/00755}

\date{}

\begin{document}


\begin{abstract}

Let $X$ be a closed semialgebraic set of dimension $k.$ 
If $n\ge 2k+1$, then   there is a bi-Lipschitz and semialgebraic  embedding of $X$ into $\R^n.$
Moreover, if   $n \ge 2k+2$, then this embedding is unique (up to a bi-Lipschitz and semialgebraic homeomorphism of $\R^n$). 
We also give  local counterparts of these results.
\end{abstract}

\bibliographystyle{alpha}

\maketitle

\section{Introduction}
There are three equivalence relations in Lipschitz Geometry of singular
sets. A semialgebraic (or algebraic) subset can be equipped with inner
metric (or in other words, the length metric) defined as minimal length of
a path, connecting two points, or with outer metric - where the distance is
defined as the distance in the ambient Euclidean space. One can define the Lipschitz
Inner equivalence, Lipschitz Outer equivalence and Lipschitz Ambient
equivalence. Two sets are called Lischitz Inner equivalent, if there exists
a bi-Lipschitz map between the two sets with respect to the Inner Metric.
The sets are called Lipschitz Outer equivalent if there exists a
bi-Lipschitz map with respect to the outer metric. Finally, the sets are
called Ambient Lipschitz equivalent if there exists a bi-Lipschitz map with
respect to the outer metric, that can be extended as a bi-Lipschitz map to the ambient space.
The equivalence relations defined above are different. The
relation between the inner and the outer metrics are very well investigated
and the theory of LNE (Lipschitz Normal embedding) is on the rapid
development (see, for example  the works of   Misef and Pichon \cite{PM} and
Kerner, Pedersen and Ruas \cite{KPR}). Our work is devoted to the relations
between the outer and ambient equivalence relations. There are many
examples of families of outer bi-Lipschitz equivalent semialgebraic sets,
such that on any pair of sets , belonging to the same family is not ambient
topological equivalent. One of the examples of such family is the knot
theory. The examples of Neumann and
Pichon show that even in the case of germs of complex surfaces the outer
Lipschitz equivalence
 does not imply the ambient topological equivalence \cite{NP}. That is why it
makes sense to study the relation of outer and ambient equivalence, only in
the case, when one has no topological obstructions. One of very basic
examples in the paper show that two compact curves in the plane can be Outer
Lipschitz equivalent, ambient topologically equivalent, but not Ambient
Lipschitz equivalent. The examples of Birbrair and Gabrielov show that for
germs of real surfaces the outer Lipschitz equivalence, considered together
with ambient topological equivalence, do not imply ambient Lipschitz
equivalences \cite{lev}.

The paper is devoted to the following questions.

Section 3. Consider a semialgebraic set. How one can estimate a dimension
of the ambient space , where the set can be bi-Lipschitzly embedded.
We prove :

\vspace{3mm}
\noindent 

{\bf Theorem 3.3.}
{\it \ Let $X$ be a closed semialgebraic   set  of dimension $k.$
Then there is a bi-Lipschitz and semi-algebraic embedding of $X$ into $\R^{2k+1}.$}

\vspace{3mm}

We give also a complex version of this theorem (see Theorem \ref{zan1}). 
Another result of this type is:

\vspace{3mm}
\noindent 

{\bf Theorem 3.5.}{\it 
\ Any compact semialgebraic set $X$  of dimension  $k$ is bi-Lipschitz equivalent, with respect to the inner metric, to a  Lipschitz normally embedded set  $Y\subset\R^{2k+1}.$}

\vspace{3mm}

Section 4 is devoted to the question of the uniqueness of the bi-Lipschitz
embedding. This question can be also reformulated as follows : when the
Outer Lipschitz equivalence implies the ambient Lipschitz equivalence.
In general two such embeddings into $\R^{2k+1}$ are not equivalent, i.e., this embedding is not unique.
We say that a semialgebraic set $X$ admits a unique embedding to $\R^n$, if for any two bi-Lipschitz embeddings
$f,g: X\to\R^n$, where $f(X), g(X)$ are semialgebraic sets,  there is a bi-Lipschitz
homeomorphism $\Phi$ of $\R^n$, such that $g=\Phi \circ f.$ In other words the subset $X\subset \R^n$ has a unique embedding into $\R^n$, if any other semialgebraic subset $Y\subset \R^n$, which is equivalent to $X$ with
respect to the outer metric, is Ambient Lipschitz equivalent to $X.$ 
In the papers \cite{jel}, \cite{top}  the third named author
 considered a problem of existence of unique
embedding for the case of smooth varieties.
Here, we  consider the same problem in the category  of
real semialgebraic sets with Lipschitz  mappings
as morphisms. We say that a bi-Lipschitz homeomorphism is said to be {\it tame} if it is a
composition of bi-Lipschitz homeomorphisms of the form $$\Phi :\R^n\ni
(x_1,...,x_n)\to (x_1,..., x_{n-1}, x_n+p(x_1,..., x_{n-1}))\in
\R^n$$  and linear mappings with determinant one. Note that such a mapping has almost everywhere jacobian equal to one, hence 
it preserves the $n-$dimensional volume of $\R^n.$
The main result of this section is:

\vspace{3mm}
\noindent 

{\bf Theorem 4.5.}{\it\
Let $X$ be a  closed semialgebraic subset of $\R^n$ of dimension
$k.$ Let $f : X \to \R^n$ be a semialgebraic and bi-Lipschitz embedding. If $n\ge 2k+2,$ then
there exists a tame bi-Lipchitz and semialgebraic homeomorphism $F : \R^n\to\R^n$ such that
$$F|_X = f.$$ Moreover, there is a continuos and semialgebraic  family of tame bi-Lipchitz and semialgebraic homeomorphisms 
$F_t: \R^n\times\R\to\R^n,$ such that $F_0=identity \ and \ F_1|_X = f.$}

\vspace{3mm}
\noindent 

{\bf Corollary 4.6.}{\it \
Let $X$ be a closed semialgebraic  set  of dimension $k.$ If $n\ge
2k+2$, then $X$ has a  unique semialgebraic and bi-Lipchitz embedding into $\R^n$ (up to a semialgebraic and bi-Lipschitz tame homeomorphism of $\R^n$).}

\vspace{3mm}

If we do not assume that the mapping $f$ is semi-algebraic our result is slightly worse (see Theorem \ref{glowne01}).

\vspace{3mm}

Section 5 is devoted to the local case. We show that similar method can be applied in the local case to give slightly better local results. 
The analog of Theorem
3.3 is the following:

\vspace{3mm}
\noindent 

{\bf Theorem 5.4.}{\it \
Let $(X,0)$ be a germ of closed semialgebraic   set  of dimension $k.$
Then there is a bi-Lipschitz and semi-algebraic embedding of $(X,0)$ into $(\R^{2k},0).$}

\vspace{3mm}

We give also a complex version of this theorem (see Theorem \ref{zan3}).
Moreover we have:

\vspace{3mm}
\noindent 

{\bf Theorem 5.9.}{\it \
Let $(X,0)$ be a  closed semialgebraic subset of $(\R^n,0)$ of dimension
$k.$ Let $f : (X,0) \to (\R^n,0)$ be a semialgebraic and bi-Lipschitz embedding. If $n\ge 2k+1,$ then
there exists a germ of tame bi-Lipchitz and semialgebraic homeomorphism $F : (\R^n,0)\to(\R^n,0)$ such that
$$F|_X = f.$$ Moreover, there is a continuos and semialgebraic  family of germs of tame bi-Lipchitz and semialgebraic homeomorphisms 
$F_t: (\R^n,0)\times\R\to(\R^n,0),$ such that $F_0=identity \ and \ F_1|_X = f.$}

\vspace{3mm}
\noindent 

{\bf Corollary 5.10.}{\it \
Let $(X,0)$ be a closed semialgebraic  set  of dimension $k.$ If $n\ge
2k+1$, then $X$ has a  unique semialgebraic and bi-Lipchitz embedding into $(\R^n,0)$ (up to a germ of semialgebraic and bi-Lipschitz tame homeomorphism of $(\R^n,0)$).}

\vspace{3mm}

As before, if we do not assume that $f$ is semi-algebraic then we obtain a slightly worse result (see Theorem \ref{zan2}). 
Finally, we show that for germs of curves in $\R^2$ we have a following weak uniqueness  of embedding:

\vspace{3mm}
\noindent 

{\bf Corollary 5.18.}{\it \
Let $(X,0)$ and $(Y,0)$ be germs of 1-dimensional semialgebraic subsets of $\R^2$. If  $(X,0)$ and $(Y,0)$ are outer bi-Lipschitz equivalent, then they are ambient bi-Lipschitz equivalent.}

\vspace{5mm}

However for surfaces the situation is more complicated. There are germ of surfaces $(X,0)$ and $(Y,0)$ in $\R^n$ (where $n=3$ or $n=4$)
and a semialgebraic bi-Lipschitz mapping $f:(X,0)\to (Y,0)$ which can be extended to homeomorphism of $(\R^n,0)$ but it can not be extended to a bi-Lipschitz  homeomorphism of $(\R^n,0)$ (see \cite{lev}).

\section{Preliminaries}
In this paper we consider the Lipschitz category $\mathcal L.$ Its objects are
closed subsets of $\R^n$ (mainly semialgebraic) equipped with the outer metric
and morphisms are  Lipschitz mappings with respect to the outer metric.
We also consider the Lipchitz semialgebraic category $\mathcal LS$. Its objects are
closed semialgebraic subsets of $\R^n$ equipped with the outer metric
and morphisms are  Lipschitz and semialgebraic  mappings with respect to the outer metric.
We start with the following basic definition which is valid for both categories (for the category $\mathcal LS$ we use brackets):

\begin{defi}\label{embedding}
Let $X$ be closed (semialgebraic) subset of $\R^n$ and let $f:X\to \R^n$ be a
Lipschitz mapping. We say that $f$ is a  bi-Lipschitz (semialgebraic) embedding if

1) $f(X)=Y$ is a  closed subset of $\R^n$,

2) the mapping $f : X\to Y$ is a bi-Lipschitz (semialgebraic) isomorphism, i.e., $f$
and $f^{-1}$ are Lipschitz (semialgebraic) mappings.
\end{defi}

\noindent Let $X$ be a closed subset of $\R^n$. We will denote by
${\bf V}(X)$ the vector space of all Lipschitz  functions on $X$. If $f :
X\to Y$ is a Lipschitz mapping of closed sets, then we have
the natural homomorphism ${\bf V}(f) : {\bf V}(Y)\ni h\to h\circ
f\in {\bf V}(X).$ 

If $X$ is a closed semialgebraic subset of $\R^n$, then by ${\bf VS}(X)$ we will denote
the vector space of all Lipschitz and semialgebraic   functions on $X$.

Since every Lipschitz function
on $X$ is the restriction of some Lipschitz function on $\R^n$ (see \cite{K}, \cite{Mc}), 
we have that the mapping $\iota^*: {\bf V}(\R^n)\ni h\mapsto h\circ \iota \in {\bf V}(X)$ is an epimorphism,
where the mapping $\iota$ is the inclusion
$\iota: X\to \R^n$. We have the following more general
fact:

\begin{pr}
Let $X\subset \R^m$ be a closed subset and    $f:X\to \R^n$ be a Lipschitz mapping.  The
following conditions are equivalent:

1) f is a bi-Lipschitz embedding,

2) the induced mapping ${\bf V}(f):{\bf V}(\R^n)\to {\bf V}(X)$ is an
epimorphism.
\end{pr}

\begin{proof}
1) $\implies$ 2) Indeed, let $Y=f(X).$ By definition  $Y$ is a
closed subset  and the mapping
$${\bf V}(f'):  {\bf V}(Y)\ni \alpha\to \alpha\circ f\in
{ {\bf V}}(X)$$ is an isomorphism (here $f'$ denotes the mapping
$f': X\ni x\to f(x)\in Y).$ Now let $\iota :Y\to \R^n$ be the inclusion.
Since every Lipschitz function $\sigma : Y\to \R^n$ can be extended to
a global Lipschitz function $\Sigma : \R^n \to \R$,  the mapping ${\bf 
V}(\iota): {\bf V}(\R^n)\to { {\bf
V}}(Y)$ is an epimorphism. But ${\bf V}(f)={\bf V}(f')\circ {\bf
V}(\iota).$

2) $\implies$ 1)   Let
$x_1,...,x_m$ be coordinates on $\R^m.$ By the assumption we can
find Lipschitz functions $H_i\in {\bf V}(\R^n)$ such that $x_i=H_i\circ
f$ (on $X$). Put $H=(H_1,...,H_m).$ We have $identity=H\circ f$, hence the mapping $f$ is proper and consequently 
the set $f(X)$ is closed.
Moreover, $f^{-1}=H$ on the set $f(X)$, and the mapping $H$ is Lipschitz.   This implies that
 the mapping $f: X\to Y$ is  bi-Lipschitz.
\end{proof}

By \cite{AF} the same result holds in the category $\mathcal LS$ (with the same proof as above):

\begin{pr}\label{semialgebraic}
Let $X\subset \R^m$ be a closed semialgebraic subset and    $f:X\to \R^n$ be a Lipschitz semialgebraic  mapping.  The
following conditions are equivalent:

1) f is a bi-Lipschitz semialgebraic embedding,

2) the induced mapping ${\bf VS}(f):{\bf VS}(\R^n)\to {\bf VS}(X)$ is an
epimorphism.
\end{pr}

\section{Extension of global bi-Lipschitz embeddings}
\noindent In this section we will prove our  main result. To do
this  we need  a series of  lemmas:

\begin{lem}\label{proj}
Let $X$ be a closed subset of $\R^n.$ Assume that the
projection $\pi : X \ni (x_1,...,x_n)\to
(x_1,...,x_{l},0,...,0)\in \R^{l}\times\{0\}$ is a bi-Lipschitz
embedding. Then, there exists a tame  bi-Lipschitz homeomorphism $\Pi :
\R^n\to \R^n$ such that $\Pi|_X=\pi.$
\end{lem}

\begin{proof}
Let $X':=\pi(X).$ It is a closed  subset of $\R^n.$ Consider
the mapping $\pi : X\to X'\subset \R^n.$ It is a bi-Lipschitz embedding, so
the mapping ${\bf V}(\pi) : {\bf V}(\R^n)\to {\bf V}(X)$ is an
epimorphism. In particular for every $s>l$ there exists a Lipschitz
function $p_s\in {\bf V}(\R^n)$ such that $x_s=p_s(x_1,...,x_l)$
(on $X$). Consider the mapping
$$\Pi (x_1,...,x_n)=(x_1,...,x_l, x_{l+1}-p_{l+1}(x_1,...,x_l),..., x_n-p_n(x_1,...,x_l)).$$
The mapping $\Pi$ is a tame bi-Lipschitz homeomorphism of $\R^n$ and
$$\Pi|_X=\pi.$$\end{proof}

The next Lemma is a  variant of the Whitney Embedding
Theorem:

\begin{lem}\label{whitney}
Let $X$ be a closed semialgebraic set of $\R^n$ of dimension $k.$ If $n>2k+1,$ then there exists a system of coordinates
$(x_1,...,x_n)$ such that the projection
$\pi : X \ni (x_1,...,x_n)\to
(x_1,...,x_{2k+1},0,...,0)\in \R^{2k+1}\times\{0\}$ is  bi-Lipschitz
embedding.
\end{lem}

\begin{proof}
Let us denote by $\pi_\infty$ the hyperplane at infinity of
$\R^n.$ Thus $\pi_\infty\cong \Bbb {RP}^{n-1}$ is a 
projective space of dimension $n-1>2k.$ For a non-zero vector
$v\in \R^n$ let $[v]$ denote the corresponding point in $\Bbb
{RP}^{n-1}.$

Let $\Delta =\{ (x,y)\in  {X}\times {X} : x=y\}.$
 Consider a mapping $$A:
{X}\times {X}\setminus \Delta \ni (x,y)\to
[x-y]\in \pi_\infty.$$ 
Since $A$ is semialgebraic mappings and  the  varieties
$X\times X\setminus \Delta$  is of
 dimension at most $2k,$ the set $\Lambda:= A(X\times X \setminus \Delta)$ is a semialgebraic
 subset of $\Bbb {RP}^{n-1}$ of dimension at most $2k.$
This means that also the semialgebraic set $\Sigma=cl(\Lambda)= closure
\ of\ \Lambda$ has  dimension at most $2k.$
 Consequently we have
$\pi_\infty \setminus \Sigma\not=\emptyset.$

Let $P\in \Bbb {RP}^{n-1}\setminus \Sigma$ and let $H\subset
\R^{n}$ be a hyperplane, which is orthogonal to the direction given by $P.$  Since $P\not \in\Lambda$, the projection $S :
{X}\ni x \to Px\cap  H\in  H$ is 
injective.   Moreover, since
$P\not\in\Sigma$ we see that $S$ is also proper (note that
$\Sigma$ contains all points at infinity of ${X}$).
Now consider the angle $a(P,S)$ between the line given by direction $P$ and a line given by a point $S\in \Sigma.$
If $P=(p_1:...:p_n)$ and $S=(s_1:...:s_n)$, then $$sin(a(P,S))=\sqrt{1-\frac{(\sum^n_{i=1} p_is_i)^2}{(\sum^n_{i=1} p_i^2)
(\sum^n_{i=1} s_i^2)}}.$$
The function $\Sigma\ni S\mapsto sin(a(P,S))\in \R_+$ is continuous and non-zero. Since the set $\Sigma$ is compact we have $sin(P,S)\ge \epsilon>0$
for every $S\in\Sigma.$ Now we see that the projection $\pi$ given by $P$ restricted to $X$ is bi-Lipschitz: 
$||x-y||\ge||{\pi}(x)-\pi(y)||\ge \epsilon ||x-y||.$

Now we can apply  induction and we obtain that the mapping $\pi : X \ni (x_1,...,x_n)\to
(x_1,...,x_{2k+1},0,...,0)\in \R^{2k+1}\times\{0\}$ is bi-Lipschitz.
\end{proof}

A direct consequence of Lemma \ref{whitney} is:

\begin{theo}\label{zan}
Let $X$ be a closed semialgebraic   set  of dimension $k.$
Then there is a bi-Lipschitz embedding of $X$ into $\R^{2k+1}.$
\end{theo}

This result is sharp. In general we cannot embed a semi-algebraic curve in a bi-Lipschitz way into $\R^2$, even if there is no topological obstruction to do it.
Indeed, we have the following example given by Edson Sampaio:

\begin{ex}
{\rm Let $X\subset\R^3$ be the semialgebraic curve given in the picture (a) below. Here the black curve is on the plane, but the red curve is transversal to the plane in two points. Thus $X$ can be embeded topologically into $\R^2$ -see the picture (b), but from  obvious  reasons it can not be embedded into $\R^2$ in a bi-Lipschitz way.}

\begin{center}
\includegraphics[scale=.4]{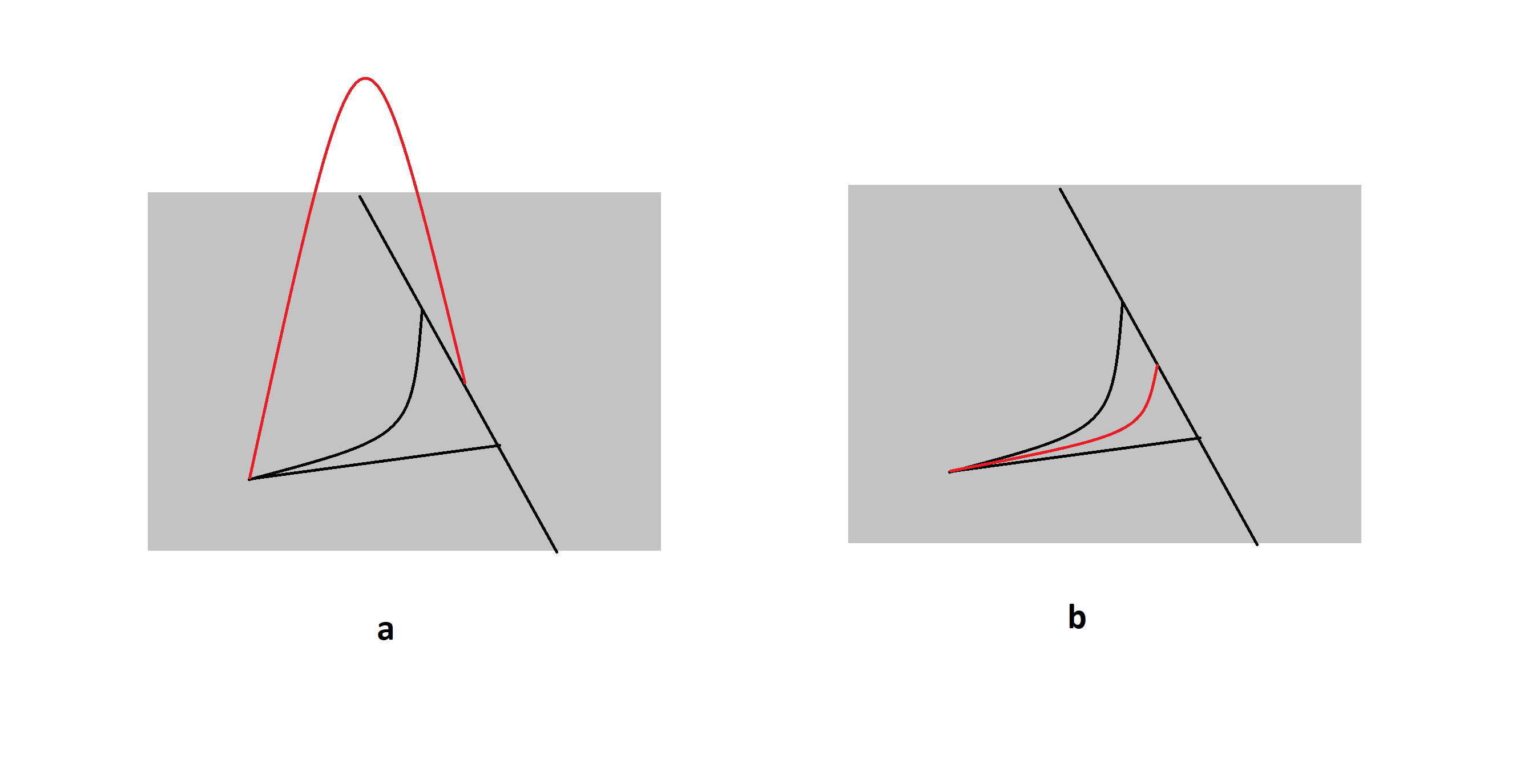}
\end{center}
\end{ex}

Moreover we have:

\begin{theo}
Any compact semialgebraic set $X$  of dimension  $k$ is bi-Lipschitz equivalent, with respect to the inner metric, to a  Lipschitz normally embedded set  $Y\subset\R^{2k+1}.$
\end{theo}

\begin{proof}
By \cite{BM} there exists $N>0$, such that $X$ is bi-Lipschitz equivalent, with respect to the inner metric, to a Lipschitz normally embedded semialgebraic set $Y\subset\R^N$. 
By Lemma \ref{proj} and Lemma \ref{whitney} there is a bi-Lipchitz homeomorphism $\Phi :\R^N\to\R^N$ such that $\Phi(Y)\subset \R^{2k+1}\times \{0\}.$ But $\Phi$ preserves the equivalence of the inner metric and the outer metric.
\end{proof}

We also need:

\begin{lem}\label{wyk}
Let $X, Y\subset \R^n$ be closed subsets. Assume that there
are linear subspaces $L^s$ and $H^{n-s},$ of dimensions $s$ and
$n-s$ respectively, such that $X\subset L^s$ and $Y\subset
H^{n-s}.$ If $f: X\to Y$ is a  bi-Lipschitz isomorphism, then there is a
tame  bi-Lipschitz homeomorphism $F: \R^n\to \R^n,$ such that $$F|_X=f.$$
\end{lem}

\begin{proof}
Note that instead of $f$ we can consider the mapping $A\circ
f\circ B,$ where $A, B$ are linear isomorphisms with determinant one. Consequently we can assume that $L^s=\R^s\times \{0\}$ and
$H^{n-s}=\{0\}\times R^{n-s}.$ Consider the closed set
$\Gamma:=graph(f).$ We can identify $\Gamma$ with a subset of
$L\times H\cong \R^n.$ Let us denote coordinates in $L$ by $x$ and
in $H$ by $y.$ We have the following diagram:
\begin{center}
\begin{picture}(240,120)(-40,40)
\put(180,160){\makebox(0,0)[tl]{$\Gamma$}}
\put(20,40){\makebox(0,0)[tl]{$X$}}
\put(180,40){\makebox(0,0)[tl]{$Y$}}
\put(190,100){\makebox(0,0)[tl]{$\pi$}}
\put(95,50){\makebox(0,0)[tl]{$f$}}
\put(80,100){\makebox(0,0)[tl]{$\psi$}}
\put(45,35){\vector(1,0){125}} \put(40,45){\vector(4,3){130}}
\put(183,145){\vector(0,-1){100}}
\end{picture}
\end{center}
\vspace{3mm} \noindent Here $\psi (x,0)=(x, f(x))$  and $\pi$ is a
projection. Since $f$ and $\psi$ are  bi-Lipschitz homeomorphisms, the
projection $\pi$ is  a bi-Lipschitz embedding. By Lemma
\ref{proj} the mapping $\pi$ has an extension to a tame  bi-Lipschitz
homeomorphism $\Pi : \R^n\to\R^n.$ Moreover,  $\psi$ is the
restriction of a tame bi-Lipschitz homemorphism of the form
$$\Psi : R^n\ni (x,y)\to (x, y+f'(x))\in \R^n$$
where $f'$ is an extension of $f$ to the whole of $L.$  In particular
$$f=\Pi\circ\Psi|_X$$
and it is enough to take $F=\Pi\circ\Psi.$
\end{proof}

Now we are in a position to prove:

\begin{theo}\label{glowne01}
Let $X,Y\subset \R^n$ be closed semialgebraic sets of dimension $k.$
Let $f: X\to Y$ be a  bi-Lipschitz homeomorphim. 
If  $n \ge 4k+2$, then $f$ can be extended to a tame  bi-Lipschitz homeomorphism $F: \R^n\to\R^n.$
Moreover, there is a continuous  family of tame  bi-Lipschitz homeomorphisms
 $F_t: \R^n\times\R\to\R^n,$ such that $F_0=identity \
and \ F_1|_X = f.$
\end{theo}

\begin{proof}
Apply Lemma \ref{whitney} to $X$ and $f(X).$ Then in virtue of
Lemma \ref{proj} we can assume that there exist tame  bi-Lipschitz homeomorphisms
$A,B:\R^n\to\R^n$ such that $A(X)\subset \R^{2k+1}\times \{0\}$
and $B(f(X))\subset \{0\}\times \R^{n-2k-1}$ (if necessary we
compose $A$ and $B$ with suitable affine transformations with
determinant $1$). Consider $f'=B\circ f\circ A^{-1};$ of course we
can assume that $f=f'.$ In particular we can assume that $X\subset
\R^{2k+1}\times \{0\}$ and $f(X)\subset \{0\}\times \R^{n-2k-1}.$ Now it is enough to apply Lemma \ref{wyk} to
obtain a tame bi-Lipschitz homeomorphism $F : \R^n\to\R^n$ such that
$$F|_X=f.$$ Moreover, there is a continous  family of tame  bi-Lipschitz
homeomorphisms $F_t: \R^n\times\R\to\R^n,$ such that $F_0=identity \
and \ F_1|_X = f.$ Indeed,  every triangular  bi-Lipschitz homeomorphism
$G=(x_1,...,x_{n-1},x_n+p(x_1,...,x_{n-1}))$ is countinously 
Lipschitz isotopic to the identity by $t\to
G_t=(x_1,...,x_{n-1},x_n+tp_n(x_1,...,x_{n-1}))$ and the same is
true for any linear  mapping with determinant  $1$ (such a mapping
is a composition of linear mappings of triangular form). 
\end{proof}

\begin{co}
Let $X$ be a closed semialgebraic 
set   of dimension $k.$ 
If  $n \ge 4k+2$, then $f$ has  a unique bi-Lipschitz embedding into $\R^n$ as semialgebraic set 
(up to a bi-Lipschitz homeomorphism of $\R^n$).
\end{co}

\begin{proof}
The set $X$ has a bi-Lipschitz embedding as semialgebraic set into $\R^n$ by Corollary \ref{zan}. It is unique by Theorem \ref{glowne01}.
\end{proof}

\section{Extension of global bi-Lipschitz, semialgebraic embedding}

Of course all results of previous section holds also in the category $\mathcal LS.$ However we show here that for this category we can obtain better results.

\begin{defi}
Let $L^s, H^{n-s-1}$ be two disjoint linear subspaces of $\Bbb P^n.$ Let $\pi_\infty$ be a hyperplane (a hyperplane at infinty) and assume that
$L^s \subset \pi_\infty.$ By a projection $\pi_L$ with center $L^s$ we mean the mapping:
$$ \pi_L : \R^n=\Bbb P^n\setminus \pi_\infty\ni x \mapsto <L^s,x> \cap H^{n-s-1}\in H^{n-s-1}\setminus \pi_\infty=\R^{n-s-1}.$$ Here by $<L,x>$
we mean a linear subspace spanned by $L$ and $\{x\}.$
\end{defi}

\begin{lem}\label{lemat}
Let $X$ be a closed subset of $\R^n.$ Denote by $\Lambda\subset \pi_\infty$ the set of directions of all secants of $X$ and let $\Sigma= cl(\Lambda)$
(where $\pi_\infty$ is a hyperplane at infinity and we consider the projective closure). Let $\pi_L:\R^n\to \R^l$ be a projection with center $L$.
Then $\pi_L|_X$ is a bi-Lipschitz embedding if and only if $L\cap \Sigma=\emptyset.$ 
\end{lem}

\begin{proof}
a) Assume that $L\cap \Sigma=\emptyset.$ We will proceed by induction.
Since a linear affine mapping is a bi-Lipschitz homeomorphism, we can assume that $\pi_L$ coincide with the projection
$\pi: \R^n\ni (x_1,...,x_n) \mapsto (x_1,...,x_k,0,...,0)\in \R^k\times \{0,...,0\}.$ We can decompose $\pi$ into two projections:  
$\pi=\pi_{2}\circ \pi_{1}$, where $\pi_1:\R^n \ni (x_1,...,x_n)\mapsto (x_1,...,x_{n-1},0)$ is a projection with a center $P_1=(0:0:...:1)$
and $\pi_2 :\R^{n-1}\ni (x_1,...,x_{n-1},0)\mapsto (x_1,...,x_{k},0,...,0)\in \R^{k}\times \{(0,...,0)\}$ is a projection with a center $L':=\{ x_0=0,...,x_{k}=0\}.$ Since $P_1\in L$ and consequently $P_1\not\in \Sigma$, we can check exactly as in the proof of lemma \ref{whitney} that $\pi_1$ is bi-Lipschitz. Now
note that if $\pi_1(X)=X'$, then $\Sigma'=\pi_1(\Sigma).$ Moreover $L'=L\cap \{ x_n=0\}$ and $<L', P_1>=L$. This means that $\Sigma'\cap L'=\emptyset.$ Now we can finish by induction.

\vspace{5mm}

b) Assume that $\pi_L|_X$ is a bi-Lipschitz and $\Sigma \cap L\not=\emptyset.$ As before we can change a system of coordinates in this way that
$\pi: \R^n\ni (x_1,...,x_n) \mapsto (x_1,...,x_k,0,...,0)\in \R^k\times \{0,...,0\}.$ Moreover, we can assume that $P_1=(0:0:...:1)\in \Sigma.$
Now we can see that $\pi_1$ is not bi-Lipschitz. Indeed there is a sequence of secants $l_n=(x_n,y_n)$ of $X$ whose directions tends to $P_1.$
Hence $||x_n-y_n||sin(a(P_1,l_n))=||\pi_1(x_n)-\pi_1(y_n)||$ and consequently $$\frac{||x_n-y_n||}{||\pi_1(x_n)-\pi_1(y_n)||}=\frac{1}{sin(a(P_1,l_n))}\to \infty.$$ Now it is enough to note that $||\pi_2(x)-\pi_2(y)||\le ||x-y||,$ hence $||\pi(x_n)-\pi_(y_n)||=||\pi_2(\pi_1(x_n))-\pi_2(\pi_1(y_n))||\le ||\pi_1(x_n)-\pi_1(y_n)||$. Thus  $$\frac{||x_n-y_n||}{||\pi(x_n)-\pi(y_n)||}\ge \frac{||x_n-y_n||}{||\pi_1(x_n)-\pi_1(y_n)||}\to \infty.$$ This contradiction finishes the proof.
\end{proof}

\begin{lem}\label{wykres}
Let $X\subset \R^{n}$ be  a closed  set and let $f: X\to\R^m$ be a Lipschitz mapping. Let $Y:=graph(f)\subset \R^n\times \R^m.$
Then the mapping $\phi: X\ni x \mapsto (x,f(x)) \in Y$ is a bi-Lipschitz embedding.
\end{lem}

\begin{proof}
Since $f$ is Lipschitz, there is a constant $C$ such that $$||f(x)-f(y)||<C||x-y||.$$
We have $$||\phi(x)-\phi(y)||=||(x-y,f(x)-f(y))||\le$$  $$ \le ||x-y||+||f(x)-f(y)||\le ||x-y||+C||x-y||\le (1+C)||x-y||.$$ 
Moreover $$||x-y|| \le ||\phi(x)-\phi(y)||.$$ Hence $$||x-y||\le ||\phi(x)-\phi(y)||\le (1+C)||x-y||.$$
\end{proof}

\begin{lem}\label{glowne}
Let $X\subset \R^{2n}$ be  a closed semialgebraic set of dimension $k$ and $n\ge 2k+1.$ 
Assume that mappings
$$\pi_1 : X\ni (x_1,...,x_n,y_1,...,y_n)\to  (x_1,...,x_n)\in
\R^n$$ and
$$\pi_2 : X\ni (x_1,...,x_n,y_1,...,y_n)\to  (y_1,...,y_n)\in \R^n$$

\noindent are bi-Lipschitz embeddings. Then there exist linear isomorphisms $S,
T $ of $\R^n$ with determinant $1$,  such that if we change coordinates in $\R^n\times
\{0\}$ by $(z_1,...,z_n)=S(x_1,...,x_n)$ and we change coordinates
in $ \{0\}\times \R^n$ by $(w_1,...,w_n)=T(y_1,...,y_n)$, then all
projections

$$ q_r : X\ni (z_1,...,z_n,w_{1},...,w_n)\to  (z_1,...,z_{r},w_{r+1},...,w_n)\in \R^n,
 \ r=0,..., n$$

\noindent are bi-Lipschitz embeddings.
\end{lem}

\begin{proof}
Let us denote by $\pi_\infty$ the hyperplane at infinity of $\R^n\times \R^n.$
Thus $\pi_\infty\cong \Bbb P^{2n-1}$ is a real projective
space of dimension $2n-1.$

We show that there are affine coordinates $(z_1,...,z_n)$ in
$\R^n\times \{0\}$ and affine coordinates $(w_1,...,w_n)$ in
$\{0\}\times \R^n$ such that all projection
$$ q_i : X\ni (z_1,...,z_n,w_{1},...,w_n)\to  (z_1,...,z_{n-i},w_{n-i+1},...,w_n)\in \R^n,
 \ i=0,..., n$$
are embeddings. On $\pi_\infty$ we have coordinates $(x:y).$ Since
the projections $\pi_1|X, \pi_2|X$ are bi-Lipschitz embedding we have by Lemma \ref{lemat} that $$\{(x:y)
\in\pi_\infty : x_1=0,...,x_n=0 \}\cap \Sigma=\emptyset$$ and  $$\{(x:y)
\in\pi_\infty : y_1=0,...,y_n=0 \}\cap \Sigma=\emptyset.$$

Since $\Sigma$ is a semialgebraic set of dimension at most $2k$, if  $H=\{ x \in \pi_\infty
: \sum^n_{i=1} c_i x_i=0\}$ is a generic hyperplane, then
$\dim H\cap \Sigma\le 2k-1.$ Indeed, the base points set of the pencil of such hyperplanes is $L=\{ x_1=0,...,x_n=0\}$
and we know that $L\cap \Sigma=\emptyset.$ Now consider a finite partition $\mathcal{S}=\{S_i\}$ of $\Sigma$ into smooth connected semialgebraic submanifolds.
Then $S_i\cap H$ is either equal to $S_i$ or this intersection has dimension at most $\dim S_i-1.$ The first possibility is excluded because in this case $S_i\subset L.$

Continuing in this fashion we see that we can choose $n$ generic
hyperplanes given by equations: $z_i=\sum^n_{k=1} a_{i,k}x_k, \
i=1,...,n$, such that for every $1\le r \le n$ the sets  $(\{ z_1=0,..., z_r=0\})\cap \Sigma$ are
semialgebraic subset of $\pi_\infty$ of
dimension at most $2k-r.$ In particular, the set $ (\{
z_1=0,..., z_{n}=0\}\cap \Sigma)$ is empty.

Now in the same way we can choose  a generic hyperplane given by
the equation $w_n=\sum^n_{k=1} b_{n,k}y_k, \ i=1,...,n$, such that
$\{ z_1=0,..., z_{n-1}=0, w_n=0\}\cap \Sigma=\emptyset$  and
additionally for every $0\le r\le n-1$ we have dim   $ \{
z_1=0,..., z_{r}=0, w_n=0\}\cap \Sigma\le 2k-r-1.$  Further we can construct
$w_{n-1}=\sum^n_{k=1} b_{n-1,k}y_k, \ i=1,...,n$, such that
$\{ z_1=0,..., z_{n-2}=0, w_{n-1}=0, w_n=0,\}\cap \Sigma=\emptyset$
 and additionally for every $0\le r\le n-2$ we
have dim   $\{ w_1=0,..., w_{r}=0, z_{n-1}=0, z_n=0\}\cap \Sigma\le
2k-r-2.$  Continuing in this manner we find a system of
coordinates $(z_1,...,z_n,w_1,...,w_n)$ we are looking for: for
all $0 \le r\le n$ we have $$\Sigma\cap \{ z_1=0,..., z_r=0,
w_{r+1}=0,..., w_n=0\}=\emptyset,$$ which implies that the mapping
$$ q_r : X\ni (z_1,...,z_n,w_{1},...,w_n)\to  (z_1,...,z_{r},w_{r+1},...,w_n)\in \R^n,
\ r=0,..., n$$
is a bi-Lipchitz embedding.  Moreover, we can always assume (by the construction)
that transformation $S: (x_1,..., x_n)\to (z_1,..., z_n)$ and $T:
(y_1,..., y_n)\to (w_1,..., w_n)$ have determinant $1.$
\end{proof}

Now we are in a position to prove the   main result of this section:

\begin{theo}\label{semialgebraic1}
Let $X$ be a  closed semialgebraic subset of $\R^n$ of dimension
$k.$ Let $f : X \to \R^n$ be a semialgebraic and bi-Lipschitz embedding. If $n\ge 2k+2,$ then
there exists a tame bi-Lipchitz and semialgebraic homeomorphism $F : \R^n\to\R^n$ such that
$$F|_X = f.$$ Moreover, there is a continuos and semialgebraic  family of tame bi-Lipchitz and semialgebraic homeomorphisms 
$F_t: \R^n\times\R\to\R^n,$ such that $F_0=identity \ and \ F_1|_X = f.$
\end{theo}

\begin{proof}
Apply Lemma \ref{whitney} to $X$ and $f(X).$ Then in virtue of
Lemma \ref{proj} we can assume that there exists tame
bi-Lipschitz homeomorphisms $A,B:\R^n\to\R^n$ such that $A(X)\subset
\R^{2k+1}\times \{0\}$ and $B(f(X))\subset \{0\}\times \R^{2k+1}$
(if necessary we compose $A$ and $B$ with suitable affine
transformations with determinants equal to $1$). Consider
$f'=B\circ f\circ A^{-1},$ of course we can assume that $f=f'.$ In
particular we can assume that $X\subset \R^{2k+1}\times \{0\}$ and
$f(X)\subset \{0\}\times \R^{2k+1}$and that $n=2k+2.$  Thus
$f=(0,f_1,...,f_{n-1}).$

Applying Lemma \ref{glowne} to the set $X'=graph(f)\subset
\R^{2k+1}\times\R^{2k+1}$ we see that there are linear
transformations $ T, S:\R^{n-1}\to\R^{n-1}$ with determinant $1$, such that if we put
$(z_1,...,z_{n-1})= S(x_1,..., x_{n-1})$ and $(f_1'(z),...,
f'_{n-1}(z))= T(f_1\circ S^{-1}(z),..., f_{n-1}\circ S^{-1}(z))$
then all mappings
$$q_r': X \ni (z_1,...,z_{n-1},0)\to (z_1,...,z_r,f'_{r+1}(z),...,f'_{n-1}(z))\in \R^{n-1}$$
are semialgebraic and bi-Lipschitz embeddings (as a composition of the semialgebraic and  bi-Lipschitz mapping
$X\to graph(f)$ with $q_r$-notation as in Lemma \ref{glowne}).

Now we construct a sequence of tame semialgebraic and bi-Lipschitz homeomorphisms $A_{n-1},
A_{n-2}, ..., A_1, A_0$ and $B_{n-1},$ $ B_{n-2}, ...,
B_1, B_0$ such that for $z\in X$ we have
$$B_r\circ A_r\circ...\circ B_{n-1}\circ A_{n-1}(x)=
(z_1,..., z_r, 0, f'_{r+1}(z),...,f'_{n-1}(z)).$$ We proceed by
induction. For $r=n-1$ it is enough to put
$A_{n-1}=B_{n-1}=identity.$ Now assume that $1\le r <n-1$ and we
have constructed a sequence of tame semialgebraic bi-Lipschitz homeomorphisms $A_r,....,
A_{n-1}$ and $B_r,..., B_{n-1}$ such that for $z\in X$ we have
$$B_r\circ A_r\circ...\circ B_{n-1}\circ A_{n-1}(x)=(z_1,...,z_r,0,f'_{r+1}(z),...,
f'_{n-1}(z)).$$ We
show how to construct $A_{r-1}$ and $B_{r-1}.$ Note that the
mapping
$$q_r': X \ni (z_1,...,z_{n-1},0)\to (z_1,...,z_r,f'_{r+1}(z),...,f'_{n-1}(z))\in \R^{n-1}$$
is a bi-Lipchitz  embedding. Consequently there exists a semialgebraic and Lipschitz function
$P_{r-1}$ (see Proposition \ref{semialgebraic}) such that $$f'_r(z)=P_{r-1}(z_1,..., z_r, f'_{r+1}(z),...,
f'_{n-1}(z)).$$ Consider a tame semialgebraic and bi-Lipchitz homeomorphism
$$A_{r-1}: \R^n\ni (t_1,...,t_n)\to (t_1,...,t_{r}, t_{r+1}+
P_{r-1}(t_1,...,t_r,t_{r+2},...,t_n),t_{r+2},...,t_n)\in \R^n.$$
Thus for $z\in X$ we have
$$A_{r-1}(z_1,...,z_r,0,f'_{r+1}(z),...,
f'_{n-1}(z))=(z_1,...,z_r,f'_r(z),...,f'_{n-1}(z)).$$ Since the
mapping $(z_1,...,z_{r-1},f'_{r}(z),..., f'_{n-1}(z))$ restricted
to $X$ is a semialgebraic and bi-Lipschitz embedding, there exists a semialgebraic and Lipchitz function $Q_{r-1}$
such that $$z_r=Q_{r-1}(z_1,...,z_{r-1},f'_{r}(z),...,
f'_{n-1}(z)).$$ Consider a tame semialgebraic and bi-Lipchitz homeomorphism
$$B_{r-1}: \R^n\ni (t_1,...,t_n)\to (t_1,...,t_{r}-Q_{r-1}(t_1,...,t_{r-1},t_{r+1},...t_n),
t_{r+1},...,t_n)\in \R^n.$$ For $z\in X$ we have
$$B_{r-1}\circ A_{r-1}(z_1,..., z_r, 0,
f'_{r+1}(z),...,f'_{n-1}(z))=(z_1,..., z_{r-1}, 0,
f'_{r}(z),...,f'_{n-1}(z)).$$ Finally by the mathematical
induction we obtain a sequence of tame semialgebraic and bi-Lipchitz homeomorphisms $A_{n-1},
A_{n-2}, ..., A_1, A_0$ and $B_{n-1}, B_{n-2}, ..., B_1, B_0$ such
that for $z\in X$ we have
$$B_0\circ A_0\circ...\circ B_{n-1}\circ A_{n-1}(z)=
( 0, f'_{1}(z),...,f'_{n-1}(z)).$$ If we take
$T_1(x_1,...,x_n)=(T(x_1,...,x_{n-1}), x_n)$ and
$S_1(y_1,...,y_n)=(y_1, S(y_2,...,y_n)),$ then
$$S_1^{-1}\circ B_0\circ A_0\circ...\circ B_{n-1}\circ
A_{n-1}\circ T_1(x)=(0,f_1(x),...,f_{n-1}(x)).$$ Now it is enough
to put $F=S_1^{-1}\circ B_0\circ A_0\circ...\circ B_{n-1}\circ
A_{n-1}\circ T_1.$ 

The last statement we prove exactly as in the proof of Theorem \ref{glowne01}.
\end{proof}

\begin{co}\label{unique}
Let $X$ be a closed semialgebraic  set  of dimension $k.$ In $n\ge
2k+2$, then $X$ has a  unique semialgebraic and bi-Lipchitz embedding into $\R^n$ (up to a semialgebraic 
and bi-Lipschitz homeomorphism of $\R^n$).
\end{co}

\begin{ex}
{\rm Let $X\subset \R^3$ be semi-algebraic curves as on the picture (a)  and $Y\subset \R^3$ as on the picture (b).
Hence $X,Y$ are bi-Lipschitz outer equivalent, however from topological reasons they are not ambient bi-Lipschitz equivalent in $\R^3$.
Hence  Theorem \ref{semialgebraic1} does not work for $n=3$ and $k=1.$
}

\begin{center}
\includegraphics[scale=.4]{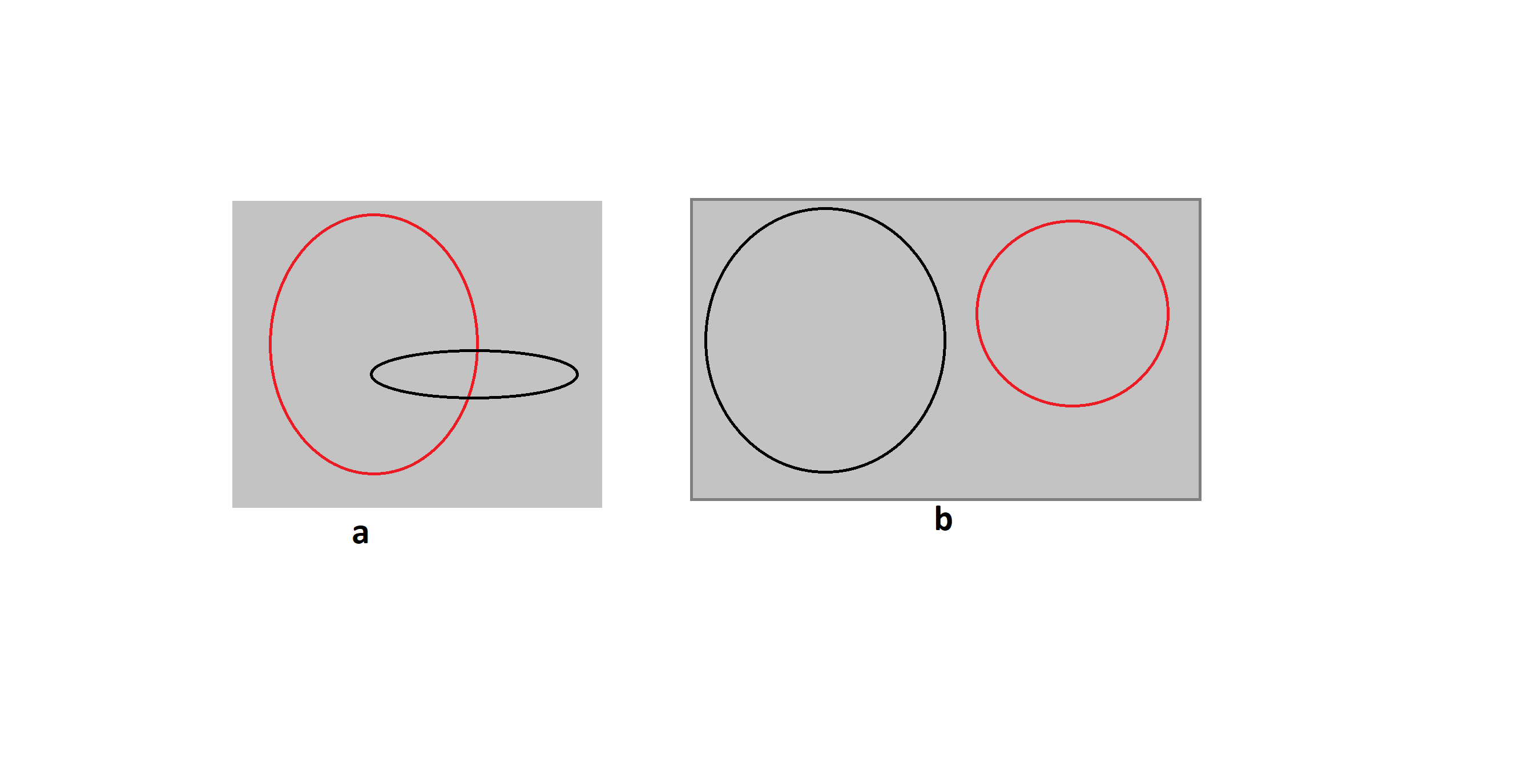}
\end{center}

{\rm It also does not work for $n=2$ and $k=1.$ Indeed,  in  $\R^2$ there are compact semialgebraic curves which are bi-Lipschitz outer equivalent and topologically ambient equivalent but not ambient bi-Lipchitz equivalent. Indeed let $X\subset \R^2$ be semi-algebraic curves as on the picture (a)  below and $Y\subset \R^2$ as on the picture (b) below. Here red curves have order of contact different than blue curves. 
Hence $X,Y$ are bi-Lipschitz outer equivalent and topologically ambient equivalent, however  they are not ambient bi-Lipschitz equivalent in $\R^2$.}

\begin{center}
\includegraphics[scale=.4]{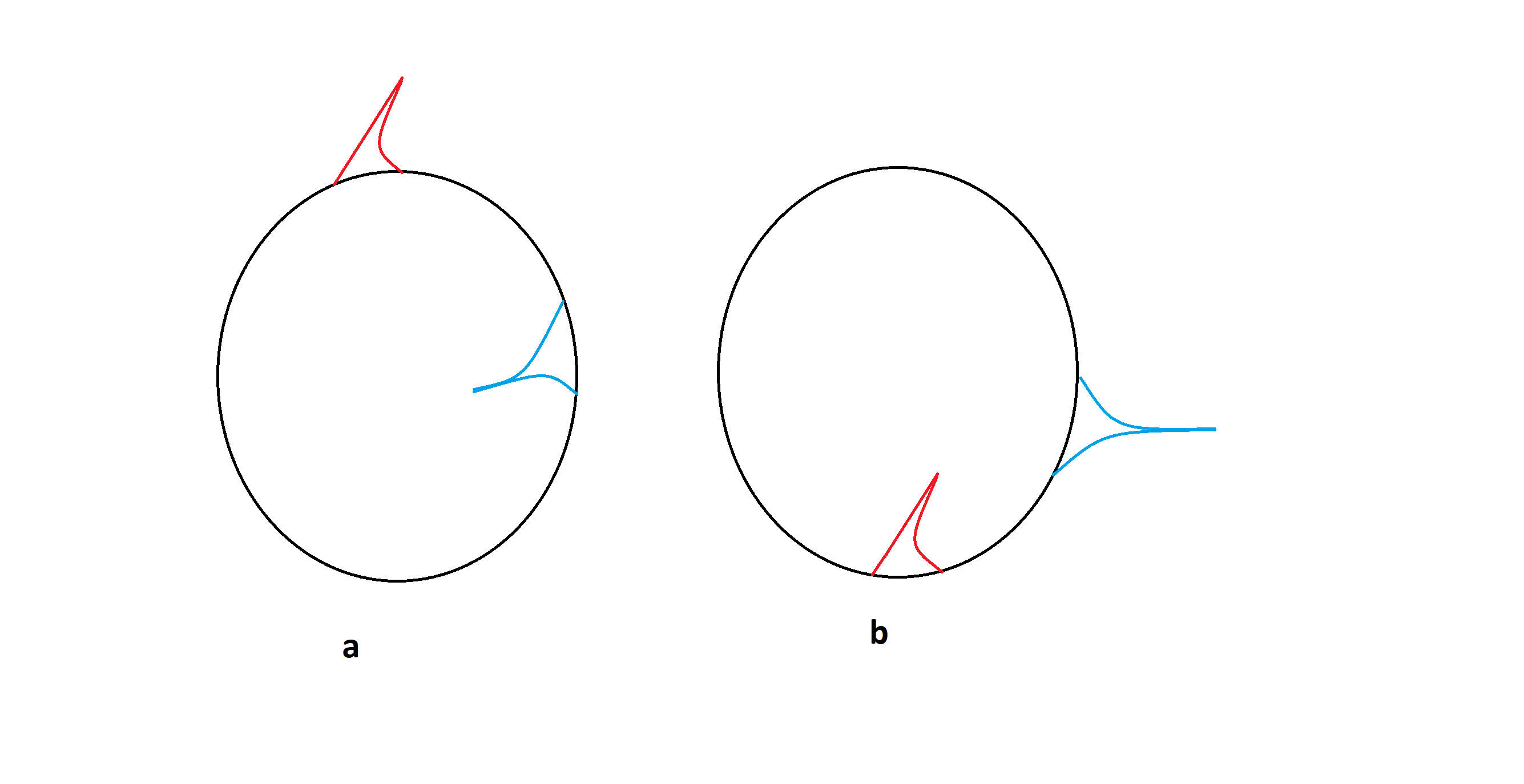}
\end{center}

\end{ex}

Hence Theorem \ref{semialgebraic1} is sharp for $k=1.$ Moreover, there are germ of surfaces $(X,0)$ and $(Y,0)$ in $\R^n$ (where $n=3$ or $n=4$)
and a semialgebraic bi-Lipschitz mapping $f:(X,0)\to (Y,0)$ which can be extended to homeomorphism of $(\R^n,0)$ but it can not be extended to a bi-Lipschitz  homeomorphism of $(\R^n,0)$ (see \cite{lev}). Hence our theorem  is nearly sharp also for $k=2.$
Let us note that we can obtain in a similar way a complex version of Theorem \ref{zan} and Theorem \ref{semialgebraic1}:

\begin{theo}\label{zan1}
Let $X$ be a complex affine algebraic  variety  of dimension $k.$
Then there is a  bi-Lipschitz (and regular) embedding  $f: X\to \C^{2k+1}$ such that the set $f(X)$ is an affine algebraic variety.
Moreover, this embedding is unique up to a bi-Lipchitz semialgebraic homeomorphism of $\C^{2k+1}.$
\end{theo}

\begin{proof}
The first part of theorem can be done similarly as in Lemma \ref{whitney}. We left details to the reader.
The second part follows directly from Theorem  \ref{semialgebraic1}.
\end{proof}

\begin{ex}{\rm
Theorem \ref{zan1}  is sharp. Indeed we show that there exists an affine curve $\Gamma$ without self-intersections, which cannot be embedded
in a bi-Lipschitz way into $\C^2$ as an algebraic curve. Note that such a curve in an obvious way has a continuous embedding into $\C^2.$ Let $\Gamma'$ be a rational curve with nine cusps (such a curve can be constructed by gluing a punctured cusp $\{x^2=y^3\}$ with a punctured $\Bbb P^1$ several times). It is a projective curve. Let $O\in \Gamma'$ be a smooth point and take $\Gamma=\Gamma'\setminus O.$ Fix an embedding $\Gamma\subset \C^N$ and consider $\Gamma$ with the induced metric structure from $\C^N.$ Hence $\Gamma$ is an affine curve with nine cusps. Assume that there is a bi-Lipschitz embedding $\phi :\Gamma \to \C^2$ and $\phi(\Gamma)=\Lambda$ is an algebraic curve. By \cite{BFLS} the curve $\Lambda$ has also nine cusps, hence the projective closure of $\Lambda$ in $\Bbb P^2$ is a rational cuspidal curve with at least nine cusps- this is a contradiction - see  \cite{tono},   Thm. 1.1, p. 216.}
\end{ex}

\section{Local case}
Here we consider the local case. 

\begin{defi}\label{embedding1}
Let $(X,0)$ be closed (semialgebraic) subset of $(\R^n,0)$ and let $f:(X,0)\to (\R^n,0)$ be a
Lipschitz mapping. We say that $f$ is a  bi-Lipschitz (semialgebraic) embedding if

1) $f((X,0))=(Y,0)$ is a germ of  closed semialgebraic subset of $\R^n$,

2) the mapping $f : (X,0)\to (Y,0)$ is a bi-Lipschitz and semialgebraic isomorphism, i.e., $f$
and $f^{-1}$ are Lipschitz and semialgebraic mappings.
\end{defi}

\noindent Let $X$ be a closed subset of $\R^n$. We will denote by
${\bf VS}_0(X)$ the vector space of all germs of Lipschitz  functions on $(X,0)$. If $f :
(X,0)\to (Y,0)$ is a Lipschitz mapping of closed sets, then we have
the natural homomorphism ${\bf VS}_0(f) : {\bf VS}_0(Y)\ni h\to h\circ
f\in {\bf VS}_0(X).$ 

If $(X,0)$ is a germ of closed semialgebraic subset of $\R^n$, then by ${\bf VS}_0(X)$ we will denote
the vector space of all germs of Lipschitz and semialgebraic   functions on $(X,0)$.

As in section 2 we have the following
fact:

\begin{pr}\label{semialgebraic3}
Let $(X,0)\subset (\R^m,0)$ be a closed semialgebraic subset and    $f:(X,0)\to (\R^n,0)$ be a Lipschitz semialgebraic  mapping.  The
following conditions are equivalent:

1) f is a bi-Lipschitz semialgebraic embedding,

2) the induced mapping ${\bf VS}_0(f):{\bf VS}_0(\R^n)\to {\bf VS}_0(X)$ is an
epimorphism.
\end{pr}

\noindent 
The next Lemma is a  local version of Lemma \ref{whitney}

\begin{lem}\label{whitney2}
Let $(X,0)$ be a closed semialgebraic set of $(\R^n,0)$ of dimension $k.$ If $n>2k,$ then there exists a system of coordinates
$(x_1,...,x_n)$ such that the projection
$\pi : X \ni (x_1,...,x_n)\to
(x_1,...,x_{2k},0,...,0)\in \R^{2k}\times\{0\}$ gives a  bi-Lipschitz
embedding of $(X,0)$ into $(\R^{2k},0)$.
\end{lem}

\begin{proof}
We can assume that $X$ is a closed semialgebraic subset of a ball $B.$
Let us denote by $\pi_\infty$ the hyperplane at infinity of
$\R^n.$ Thus $\pi_\infty\cong \Bbb {RP}^{n-1}$ is a 
projective space of dimension $n-1\ge 2k.$ For a non-zero vector
$v\in \R^n$ let $[v]$ denote the corresponding point in $\Bbb
{RP}^{n-1}.$

Let $\Delta =\{ (x,y)\in  {X}\times {X} : x=y\}.$
 Consider a mapping $$A:
{X}\times {X}\setminus \Delta \ni (x,y)\to
[x-y]\in \pi_\infty.$$ 
Let $\Gamma'$ be the graph of $A$ in $B\times B\times {\Bbb {RP}^{n-1}}$ and take $\Gamma:=\overline{\Gamma'}$
(we take this closure in $ B\times B\times {\Bbb {RP}^{n-1}}$).  Let 
$p:\Gamma \to {\Bbb {RP}^{n-1}}$, 
$q:\Gamma \to B\times B$ be the projections.
Note that $q^{-1}(0,0)=Y$ is a semialgebraic set of dimension at most $2k-1.$ Let $Y'=p(Y).$ This is a closed subset of $\Bbb {RP}^{n-1}$ of dimension  at most $2k-1$, hence   $\Bbb {RP}^{n-1}\setminus Y' \not=\emptyset.$ Let $Q\in {\Bbb {RP}^{n-1}}\setminus Y'$, since $Y'$ is closed there exists a small ball $B_1\subset \Bbb {RP}^{n-1}$ with center at $Q$ such that $B_1\cap Y'=\emptyset.$

Let $V(\epsilon)$ be a small ball in $B\times B$ and let $R_\epsilon=q^{-1}(V(\epsilon)).$ Hence  $R_\epsilon$ is a neighborhood of $Y'$ in
$\Bbb {RP}^{n-1}$. We show that for $\epsilon$ sufficiently small, if $x,y\in V(\epsilon),\ x\not=y$, then $A(x,y)\not\in B_1.$
Indeed, in other case take $\epsilon=1/n \to 0.$ Hence for every $n\in \Bbb N$ we have points $x_n,y_n$ such that $A(x_n,y_n)\in B_1.$
But then $x_n,y_n\to 0$ and $\lim A(x_n,y_n)=\lim p((x_n,y_n, [x_n-y_n]))\in p(Y)=Y'.$ It is a contradiction. 

Hence there is an $\epsilon$ sufficiently small, such that  if  $x,y\in X\cap B(0,\epsilon):=X_\epsilon, \ x\not=y$, then $A(x,y)\not\in B_1.$ Let $\Sigma=A(X_\epsilon\times X_\epsilon\setminus \Delta).$ Then $Q\not\in \overline{\Sigma}.$
Now we finish exactly as in the proof of Lemma \ref{whitney}.
\end{proof}

A direct consequence of Lemma \ref{whitney2} is:

\begin{theo}\label{zan2}
Let $(X,0)$ be a germ of closed semialgebraic   set  of dimension $k.$
Then there is a bi-Lipschitz (and semialgebraic) embedding of $(X,0)$ into $(\R^{2k},0).$
\end{theo}

Let us note that we can obtain in a similar way a complex version of Theorem \ref{zan2}:

\begin{theo}\label{zan3}
Let $(X,0)$ be a germ of closed complex algebraic  affine variety  of dimension $k.$
Then there is a bi-Lipschitz (and regular) embedding of $f:(X,0)\to(\C^{2k},0),$ such that $(f(X),0)$ is a germ of complex algebraic variety.
\end{theo}

Now we are in a position to prove the   counterpart of  Theorem \ref{glowne01}. The proof is similar  and we left it to the reader.

\begin{theo} 
Let $(X,0)$ be a  closed semialgebraic subset of $(\R^n,0)$ of dimension
$k.$ Let $f : (X,0) \to (\R^n,0)$ be a  bi-Lipschitz embedding. If $n\ge 4k,$ then
there exists a germ of tame bi-Lipchitz  homeomorphism $F : (\R^n,0)\to(\R^n,0)$ such that
$$F|_X = f.$$ Moreover, there is a continuos and semialgebraic  family of germs of tame bi-Lipchitz and semialgebraic homeomorphisms 
$F_t: (\R^n,0)\times\R\to(\R^n,0),$ such that $F_0=identity \ and \ F_1|_X = f.$
\end{theo}

\begin{co}
Let $(X,0)$ be a germ of closed semialgebraic 
set   of dimension $k.$ 
If  $n \ge 4k$, then $f$ has  a unique bi-Lipschitz embedding into $\R^n$ as semialgebraic set 
(up to a germ of a bi-Lipschitz homeomorphism of $\R^n$).
\end{co}

In the sequel we need following version of Lemma \ref{glowne}:

\begin{lem}\label{glowne11}
Let $(X,0)\subset (\R^{2n},0)$ be  a germ of a closed semialgebraic set of dimension $k$ and $n\ge 2k.$ 
Assume that mappings
$$\pi_1 : (X,0)\ni (x_1,...,x_n,y_1,...,y_n)\to  (x_1,...,x_n)\in
(\R^n,0)$$ and
$$\pi_2 : (X,0)\ni (x_1,...,x_n,y_1,...,y_n)\to  (y_1,...,y_n)\in (\R^n,0)$$

\noindent are bi-Lipschitz embeddings. Then there exist linear isomorphisms $S,
T $ of $\R^n$ with determinant $1$,  such that if we change coordinates in $\R^n\times
\{0\}$ by $(z_1,...,z_n)=S(x_1,...,x_n)$ and we change coordinates
in $ \{0\}\times \R^n$ by $(w_1,...,w_n)=T(y_1,...,y_n)$, then all
projections

$$ q_r : (X,0)\ni (z_1,...,z_n,w_{1},...,w_n)\to  (z_1,...,z_{r},w_{r+1},...,w_n)\in (\R^n,0),
 \ r=0,..., n$$

\noindent are bi-Lipschitz embeddings.
\end{lem}

\begin{proof}
Let us denote by $\pi_\infty$ the hyperplane at infinity of $\R^n\times \R^n.$
Thus $\pi_\infty\cong \Bbb P^{2n-1}$ is a real projective
space of dimension $2n-1.$

We show that there are affine coordinates $(z_1,...,z_n)$ in
$\R^n\times \{0\}$ and affine coordinates $(w_1,...,w_n)$ in
$\{0\}\times \R^n$ such that all projection
$$ q_i : (X,0)\ni (z_1,...,z_n,w_{1},...,w_n)\to  (z_1,...,z_{n-i},w_{n-i+1},...,w_n)\in (\R^n,0),
 \ i=0,..., n$$
are embeddings. On $\pi_\infty$ we have coordinates $(x:y).$ 

Let $X\subset B(0,\epsilon)$ be so small closed representative of the germ $(X,0)$  such that 
projections ${\pi_1}_{|X} , {\pi_2}_{|X}$ are bi-Lipschitz and 
 we have by Lemma \ref{lemat} that $$\{(x:y)
\in\pi_\infty : x_1=0,...,x_n=0 \}\cap \Sigma=\emptyset$$ and  $$\{(x:y)
\in\pi_\infty : y_1=0,...,y_n=0 \}\cap \Sigma=\emptyset.$$
Let $\Delta =\{ (x,y)\in  {X}\times {X} : x=y\}.$
 Consider a mapping $$A:
{X}\times {X}\setminus \Delta \ni (x,y)\to
[x-y]\in \pi_\infty.$$ 
Let $\Gamma'$ be the graph of $A$ in $B\times B\times {\Bbb {RP}^{n-1}}$ and take $\Gamma:=\overline{\Gamma'}$
(we take this closure in $ B\times B\times {\Bbb {RP}^{n-1}}$).  Let 
$p:\Gamma \to {\Bbb {RP}^{n-1}}$, 
$q:\Gamma \to B\times B$ be the projections.
Note that $q^{-1}(0,0)=Y$ is a semialgebraic set of dimension at most $2k-1.$ Let $Y'=p(Y).$ This is a closed subset of $\Bbb {RP}^{n-1}$ of dimension  at most $2k-1$.

Since $\Sigma$ is a semialgebraic set of dimension at most $2k$, if  $H=\{ x \in \pi_\infty
: \sum^n_{i=1} c_i x_i=0\}$ is a generic hyperplane, we see that
$\dim H\cap \Sigma\le 2k-1$ and $\dim H\cap Y'\le 2k-2.$ 

Continuing in this fashion we see that we can choose $n$ generic
hyperplanes given by equations: $z_i=\sum^n_{k=1} a_{i,k}x_k, \
i=1,...,n$, such that for every $1\le r \le n$ the sets  $(\{ z_1=0,..., z_r=0\})\cap \Sigma$ are
semialgebraic subset of $\pi_\infty$ of
dimension at most $2k-r.$ In particular, the set $ \{
z_1=0,..., z_{n}=0\}\cap Y'$ is empty and the set $S:= \{
z_1=0,..., z_{n}=0\}\cap \Sigma$ is at most finite. However, since $S$ is disjoint from $Y'$,  we contract  $X$ and 
 assume that also the set $S$ is empty.

Now in the same way we can choose  a generic hyperplane given by
the equation $w_n=\sum^n_{k=1} b_{n,k}y_k, \ i=1,...,n$, such that
$\{ z_1=0,..., z_{n-1}=0, w_n=0\}\cap Y'=\emptyset$  and  $w_n=\sum^n_{k=1} b_{n,k}y_k, \ i=1,...,n$, such that
$\{ z_1=0,..., z_{n-1}=0, w_n=0\}\cap \Sigma$ is finite, in fact we can assume as before that  this set is empty.

Additionally for every $0\le r\le n-1$ we have dim   $ \{
z_1=0,..., z_{r}=0, w_n=0\}\cap Y'\le 2k-r-2.$ and dim   $ \{
z_1=0,..., z_{r}=0, w_n=0\}\cap \Sigma\le 2k-r-1.$  
Further we can construct
$w_{n-1}=\sum^n_{k=1} b_{n-1,k}y_k, \ i=1,...,n$, such that
$\{ z_1=0,..., z_{n-2}=0, w_{n-1}=0, w_n=0,\}\cap \Sigma=\emptyset$
 and additionally for every $0\le r\le n-2$ we
have dim   $\{ w_1=0,..., w_{r}=0, z_{n-1}=0, z_n=0\}\cap \Sigma\le
2k-r-2.$  Continuing in this manner we find a system of
coordinates $(z_1,...,z_n,w_1,...,w_n)$ we are looking for: for
all $0 \le r\le n$ we have $$\Sigma\cap \{ z_1=0,..., z_r=0,
w_{r+1}=0,..., w_n=0\}=\emptyset,$$ which implies that the mapping
$$ q_r : (X,0)\ni (z_1,...,z_n,w_{1},...,w_n)\to  (z_1,...,z_{r},w_{r+1},...,w_n)\in (\R^n,0),
\ r=0,..., n$$
is a bi-Lipchitz embedding.  Moreover, we can always assume (by the construction)
that transformation $S: (x_1,..., x_n)\to (z_1,..., z_n)$ and $T:
(y_1,..., y_n)\to (w_1,..., w_n)$ have determinant $1.$
\end{proof}

Now we can repeat the proof of Theorem \ref{semialgebraic1} to get:

\begin{theo}\label{semialgebraic2}
Let $(X,0)$ be a  germ of a closed semialgebraic subset of $\R^n$ of dimension
$k.$ Let $f : (X,0) \to (\R^n,0)$ be a semialgebraic and bi-Lipschitz embedding. If $n\ge 2k+1,$ then
there exists a germ of tame bi-Lipchitz and semialgebraic homeomorphism $F : (\R^n,0)\to(\R^n,0)$ such that
$$F|_X = f.$$ Moreover, there is a continuos and semialgebraic  family of germs of tame bi-Lipchitz and semialgebraic homeomorphisms 
$F_t: (\R^n,0)\times\R\to(\R^n,0),$ such that $F_0=identity \ and \ F_1|_X = f.$
\end{theo}

\begin{co}\label{unique2}
Let $(X,0)$ be a germ of a closed semialgebraic  set  of dimension $k.$ If $n\ge
2k+1$, then $(X,0)$ has a  unique semialgebraic and bi-Lipchitz embedding into $(\R^n,0)$ (up to a germ of a semialgebraic 
and bi-Lipschitz homeomorphism of $(\R^n,0)$).
\end{co}

We show below that  Theorem \ref{semialgebraic2} is sharp for $k=1.$ Moreover, there are germ of surfaces $(X,0)$ and $(Y,0)$ in $\R^n$ (where $n=3$ or $n=4$)
and a semialgebraic bi-Lipschitz mapping $f:(X,0)\to (Y,0)$ which can be extended to homeomorphism of $(\R^n,0)$ but it can not be extended to a bi-Lipschitz  homeomorphism of $(\R^n,0)$ (see \cite{lev}). Hence our theorem  is  sharp also for $k=2.$

\begin{ex}
{\rm Let $X\subset\R^2$ be the semialgebraic curve given in the picture (a) below. Here the red curve has the same  contact exponents with black curves  as the blue curve. Let $f$ be the identity, and $g$ exchange the blue curve with a red one. The mapping $g$ is bi-Lipschitz (see Theorem \ref{alex} below). From  obvious  reasons there is no any germ of homeomorphism $\Phi$ of $(\R^2,0)$ such that $\Phi\circ f=g.$}

\begin{center}
\includegraphics[scale=.4]{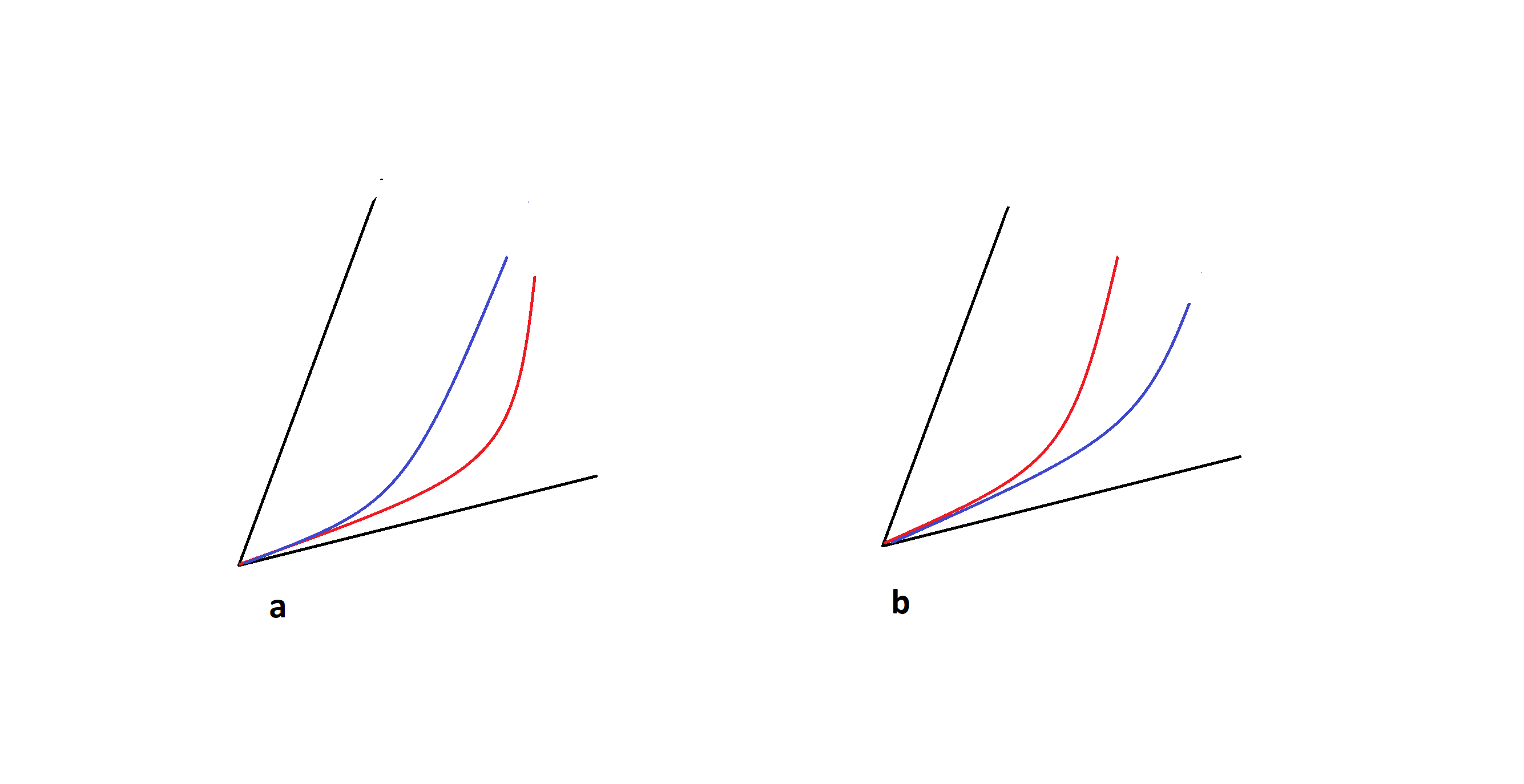}
\end{center}

{\rm However we show below (see Theorem \ref{local curves}) that an embedding of a germ of a curve into $(\R^2,0)$ has a weak uniqueness property: if $(Y_1,0)$, $(Y_2,0)$ are outer equivalent
then also $(Y_1,0)$, $(Y_2,0)$ are ambient equivalent.}
\end{ex}

The rest of the section is devoted to the proof  of Corollary 5.18.
Notice that, a proof of Corollary 5.18 consists to show that, for any pair of germs of 1-dimensional semialgebraic subsets of $\R^2$ at $0\in\R^2$, let us say  $(X,0)$ and $(Y,0)$, there exists a bi-Lipschitz homeomorphism $F\colon (\R^2,0)\rightarrow (\R^2,0)$ such that $F(X,0)=(Y,0)$ if, and only if,  the germs $(X,0)$ and $(Y,0)$ are bi-Lipschitz homeomorphic. So, before starting to prove Corollary 5.18, let us describe a discrete complete invariant of germs of 1-dimensional semialgebraic sets up to bi-Lipschitz homeomorphisms.

 Let $(X,0)$ be the germ of a 1-dimensional semialgebraic subset of $\R^n$ at $0\in\R^n$. Let us denote by $X_1,\dots,X_k$ its half-branches. For each pair of half-braches, $X_i$ and $X_j$, let us consider their {\it contact exponent} defined in the following way:
$$cont(X_i,X_j)=\lim_{r\to 0^+}\frac{\log dist(X_i\cap S(0,r), X_j\cap S(0,r))}{\log r}$$ where $S(0,r)$ denotes the Euclidean sphere in $\R^n$ with radius $r$ and centered at $0\in\R^n$. The main result of the paper \cite{BF} shows that the rational numbers $cont(X_i,X_j)$ define a complete bi-Lipschitz invariant of the germ $(X,0)$ in the following sense:

\begin{theo}[\cite{BF}, Theorem 4.1]\label{alex}
Let $(X,0)$ and $(Y,0)$ be germs of 1-dimensional semialgebraic subsets of $\R^n$ at $0\in\R^n$. The following statetments are equivalent:
\begin{enumerate}
\item There exists a bi-Lipschitz homeomorphism $f\colon (X,0)\rightarrow (Y,0)$;
\item There exists a bijection $X_i \leftrightarrow Y_{\sigma(i)}$ where $X_1,\dots,X_k$ and $Y_1,\dots ,Y_k$ are the half-branches of $(X,0)$ and $(Y,0)$, respectively, in such a way that $cont(X_i,X_j)=cont(Y_{\sigma(i)},Y_{\sigma(j)})$. In such a case, we say that the half-branches of $(X,0)$ and $(Y,0)$ have the same contact exponents.
\end{enumerate}
\end{theo}

Let us fix some notation. Let $Q\subset\R^2$ denote the first quadrant $\{(x,y) \ : \ x\geq 0 \ \mbox{and} \ y\geq 0\}$. Given a nonnegative Lipschitz semialgebraic function $f\colon [0,\delta]\rightarrow\R$; $f(0)=0$, there exist a rational number $\alpha\geq 1$ and a semialgebraic continuous function $h(x)$ such that $f(x)=x^{\alpha}h(x)$; $h(0)>0$. Such a decomposition of the funtion $f$ it will be called \emph{Puiseux decomposition of $f$ at} $0$ and $\alpha$ will be called \emph{the first Puiseux exponent of $f$ at} $0$.

\begin{re}\label{remark1} Let  $f\colon [0,\delta]\rightarrow\R$ be a nonnegative Lipshitz semialgebraic function;  $f(0)=0$. Suppose that $f(x)=x^{\alpha}h(x)$ is the Puiseux decomposition of $f$ at 0.
Let us consider the semialgebraic  map-gem $F\colon (\R^2,0)\rightarrow(\R^2,0)$ defined by
$$ F(x) = \left\{ 
\begin{array}{rcl} 
(x,\frac{y}{h(x)}) \hspace{1.65cm}& \mbox{if} & 0\leq x,\  0\leq y\leq f(x)\\
 (x,y+x^{\alpha}-f(x)) & \mbox{if} & 0\leq x, \ f(x)\leq y\\ 
(x,y)\hspace{2.1cm} & \mbox{if} & (x,y)\not\in Q 
\end{array} \right. $$
It is proved in \cite{B} that $F$ is a germ of a semialgebraic bi-Lipschitz homeomorphism from $(\R^2,0)$ onto $(\R^2,0)$ which  coincides with identity mapping outside $Q$ and $F(G_f,0)=(G_{\alpha},0)$, where $G_f$ and $G_{\alpha}$ are the graph of the functions $f(x)$ and $x^{\alpha}$, respectively.
\end{re}

As a consequence of the construction described in Remark \ref{remark1}, we get the following result.

\begin{lem}\label{lemma1} Let  $f_i,g_i\colon [0,\delta]\rightarrow\R$ be a nonnegative Lipshitz semialgebraic functions;  $f_i(0)=g_i(0)=0, \ i=1,\dots,m$. Let $X=\bigcup G_{f_i}$ and $Y=\bigcup G_{g_i}$. If the half-branches of $(X,0)$ and $(Y,0)$ have the same contact exponents, then there exists a germ of a semialgebraic bi-Lipschitz homeomorphism  $F\colon(\R^2,0)\rightarrow(\R^2,0)$ which  coincides with identity mapping outside $Q$ and $F(X,0)=(Y,0)$
\end{lem}

\begin{proof}
Let us prove it by induction on the number of half-branches $m$. Before starting this proof, let us do some comments. Without loss generality, it is clear that we may suppose that $ 0 = f_m < \cdots < f_1$ and $0 = g_m < ... < g_1.$  Case $m = 2$ follows from 
Remark \ref{remark1}. Let us suppose that $ m > 2$ and, by induction, let us suppose that our lemma holds true for semialgebraic germs with $m-1$ half-branches as stated above. Again, by using Remark \ref{remark1}, we may assume that $f_{ m -1} = x^{\alpha} = g_{m-1}$ with $\alpha\geq 1$. Let $\tilde{X}=\bigcup G_{\tilde{f}_i}$ and $\tilde{Y}=\bigcup G_{\tilde{g}_i}$ , where $\tilde{f}_i = f_i-x^{\alpha}$ and $\tilde{g}_i = g_i-x^{\alpha}$, $i = 1, ..., m-1.$ By induction, there exists a germ of semialgebraic bi-Lipschitz homeomorphism $\tilde{ F}\colon (\R^2, 0)\rightarrow (\R^2, 0)$ such that $\tilde{F}(\tilde{X},0) = (\tilde{Y},0)$ and $\tilde{F}$  coincides with identity mapping outside $Q$. Let us take $\Phi\colon (\R^2,0)\rightarrow  (\R^2, 0)$ defined by $\Phi(x,y)= (x,y - x^{\alpha})$. Thus, $F= {\Phi}^{-1}\circ\tilde{F}\circ\Phi\colon (\R^2,0)\rightarrow(\R^2,0)$ is a germ of a semialgebraic bi-Lipschitz homeomorphism such that $F(X,0) = (Y, 0)$  and  $F$ coincides with identity mapping outside $Q$.
\end{proof}

\begin{re} Let $c > 0$ and $K = \{ ( x,y )\in\R^2 \ : \  0 < x, \  0<y < cx \}.$ Then, $ (\R^2,0)\rightarrow (\R^2,0)$ given by $(x,y) \mapsto (cx - y,y)$  is a semialgebraic  bi-Lipschitz homeomorphism  and it sends $K$ onto $Q$. This allows us to rewrite Lemma \ref{lemma1} in the following way.
\end{re}

\begin{lem}\label{lemma2} Let  $f_i,g_i\colon [0,\delta]\rightarrow\R$ be a nonnegative Lipshitz semialgebraic functions;  $f_i(0)=g_i(0)=0, \ i=1,\dots,m$. Let $X=\bigcup G_{f_i}$ and $Y=\bigcup G_{g_i}$. Let $c > 0$ and $K = \{ ( x,y )\in\R^2 \ : \  0 < x, \  0<y < cx \}$ such that $K$ contains $(X,0)$ and $(Y,0)$.  If the half-branches of $(X,0)$ and $(Y,0)$ have the same contact exponents, then there exists a germ of a semialgebraic bi-Lipschitz homeomorphism  $F\colon(\R^2,0)\rightarrow(\R^2,0)$ which  coincides with identity mapping outside $K$ and $F(X,0)=(Y,0)$
\end{lem}

\begin{theo}\label{local curves}
Let $(X,0)$ and $(Y,0)$ be germs of 1-dimensional semialgebraic subsets of $\R^2$. If the half-branches of $(X,0)$ and $(Y,0)$ have the same contact exponents, then there exists a germ of semialgebraic bi-Lipschitz homeomorphism $ F \colon (\R^2,0)\rightarrow (\R^2,0)$ such that $F(X,0) = (Y,0)$.
\end{theo}

\begin{proof} Let  $X = \bigcup_i Xi$ and 
$Y = \bigcup_i Y_i$ be the half-branches of $(X,0)$ and $(Y,0)$, respectively.  Let us suppose that $cont(X_i,X_j)=cont(Y_i,Y_j)$ for any pair $i\neq j$. We claim that there exists a germ of a bi-Lipschitz semi-algebraic homeomorphism $ F \colon (\R^n,0)\rightarrow (\R^n,0)$ such that $F(X,0) = (Y,0)$. Let us prove such a claim in the case where  ambient Euclidean space is $\R^2$ .  Let  us suppose that  $(X,0)$ and $(Y,0)$ are subsets of  $\R^2$. Since $(X,0)$ and $(Y,0)$ have the same contact exponents, they have same number $t$ of tangent half-line. For each tangent half-line $L_k^{X}$ (respect. $L_k^{Y}$), let $K_k^X$ (respect. $K_k^{Y}$) be a conical region in $\R^2$ centered in $L_k^{X}$ (respect. $L_k^{Y}$) such that $K_k^X\cap K_{\tilde{k}}^X=\emptyset$ (respect. $K_k^Y\cap K_{\tilde{k}}^Y=\emptyset$) and, also, any two half-branches of $(X,0)$ (respectively $(Y,0)$) have the same unit vector tangent if, and only if, they are in the same cone. Let us re-ordering (if necessary) the cones in such a way that we come with $(X\cap K_k^X,0)$ and $(Y\cap K_k^Y,0)$ having the same contact exponents for every $k=1,\dots ,t$. Now, for each $k$,  it comes from Lemma \ref{lemma1} that there exists a bi-Lipschitz homeomorphism $F_k\colon(\R^2,0)\rightarrow(\R^2,0)$ which  coincides with identity mapping outside $K_k$ and $F(X\cap K_k^X,0)=(Y\cap K_k^Y,0)$. Let us define $F\colon(\R^2,0)\rightarrow(\R^2,0)$ by 
$F(x,y)=F_k(x,y)$ if $(x,y)\in K_k^X$ and $F(x,y)=(x,y)$ if $(x,y)\not\in \bigcup K_k^X$. We have that $F$ is a germ of a semialgebraic bi-Lipschitz homeomorphism such that $F(X,0)=(Y,0)$.
\end{proof}

\begin{co}\label{local curves1}
Let $(X,0)$ and $(Y,0)$ be germs of 1-dimensional semialgebraic subsets of $\R^n$. If  $(X,0)$ is outer Lipschitz equivalent to $(Y,0)$,
then  $(X,0)$ is ambient Lipschitz equivalent to $(Y,0)$.
\end{co}

    \end{document}